\documentclass[12pt, twoside]{article}
\usepackage{PRIMEarxiv}

\raggedright

\usepackage{setspace}

\usepackage[square,numbers]{natbib}
 \bibpunct[, ]{(}{)}{,}{a}{}{,}%

 \usepackage[title]{appendix}

\usepackage[utf8]{inputenc} % allow utf-8 input
\usepackage[T1]{fontenc}    % use 8-bit T1 fonts
\usepackage{hyperref}       % hyperlinks
\usepackage{url}            % simple URL typesetting
\usepackage{booktabs}       % professional-quality tables
\usepackage{amsfonts}       % blackboard math symbols
\usepackage{nicefrac}       % compact symbols for 1/2, etc.
\usepackage{microtype}      % microtypography
\usepackage{lipsum}
\usepackage{fancyhdr}       % header
\usepackage{graphicx}       % graphics
\graphicspath{{images/}}     % organize your images and other figures under media/ folder

\usepackage{amssymb}
\usepackage{mathtools}
\usepackage{amsmath}
\usepackage{amsthm}

\usepackage{theoremref}

\newtheorem{proposition}{Proposition}
\newtheorem{lemma}{Lemma}

\newtheorem{corollary}{Corollary}

\usepackage{algorithm}
\usepackage{algorithmic}
\usepackage[svgnames,table,]{xcolor}
\usepackage{xltabular}
\usepackage{adjustbox}
\usepackage{subcaption}
\usepackage{caption}
\usepackage{multirow}
\newcolumntype{C}{>{\centering\arraybackslash}X}
\newcolumntype{L}{>{\raggedright\arraybackslash}X}
\newcolumntype{Z}{>{\hsize=1\hsize\raggedright}X}
\newcolumntype{Y}{>{\centering\small\arraybackslash}X}
\DeclareUnicodeCharacter{0308}{ }
\DeclareUnicodeCharacter{0327}{ }

\title{A Spatial Branch-and-Bound Approach for Maximization of Non-Factorable Monotone Continuous Submodular Functions
}

\author{
  Hugh Medal, Izuwa Ahanor \\
  Department of Industrial and Systems Engineering, \\
  University of Tennessee-Knoxville\\
  \texttt{hmedal@utk.edu,iahanor@vols.utk.edu} \\
}

\begin{document}
\maketitle

\onehalfspacing

\begin{abstract}
In contrast to the many continuous global optimization methods that assume the objective function and constraints are factorable, we study how to find globally maximal solutions to problems that are not factorable, focusing on a particular class of problems. Specifically, we develop a method for non-decreasing continuous submodular functions subject to constraints. We characterize the hypograph of such functions and develop a cutting plane algorithm that finds approximate solutions and bounds using an approximation of the convex hull of the hypograph. We also test a spatial branch-and-bound approach that utilizes the approximate cutting plane algorithm to form an outer approximation and obtain upper bounds for a sub-rectangle and compare our method with a state-of-the-art commercial solver. The main result is that for some problems the property of submodularity is more useful than factorability.
\end{abstract}

% keywords can be removed

\section{Introduction}
\label{sec:1introduction}

Consider a continuous function $F:\mathcal{X} \rightarrow \mathbb{R}$ with $\mathcal{X} = \prod_{i=1}^{n}\mathcal{X}_{i}$, where $\mathcal{X}_i$ is a compact subset of $\mathbb{R}$. Many approaches for finding a global maximum/minimum of a non-convex continuous function $F$ require $F$ to be \textit{factorable}. This property allows a nonlinear expression to be decomposed into atomic units (e.g., $x$, $\sin(x)$, etc.) and operators $(\times,\div,+,-)$ linked together using an expression tree. For example, the function $F(x_1,x_2)=\sin(x_1)\,x_2$ can be factored into the atomic units $\sin(x_1)$ and $x_2$ that are linked by the multiplication operator. Then the atomic units as well as product and quotient terms can be approximated above (below) using concave (convex) envelopes. These envelopes can be used to efficiently compute an upper (lower) bound that can be used within a spatial branch-and-bound approach.

In contrast, the method that we describe in this paper does not assume that $F$ is factorable. That is, instead of relying on a factorable expression to derive envelopes, we present an approach for using \textit{submodularity} to derive them.

Toward this end we assume that $F$ is diminishing-returns (DR) submodular (i.e., submodular and component-wise concave), non-decreasing, and differentiable.

The goal of this paper is develop a method for globally maximizing $F$ subject to constraints. In other words, to find globally optimal solutions to:

\begin{equation}
	\label{opt-problem}
	\max_{x\in X} F(x),
\end{equation}

where $X\subseteq \mathcal{X}$ is defined by upper and lower bounds for each variable as well as a linear system $Ax\leq b$.

\subsection{Motivation}
There are a number of important applications that involve maximizing a continuous non-decreasing DR-submodular function. One example is when $F$ is a non-concave quadratic function of the form $F(x) = c^T x + \frac{1}{2}x^T H x + h$, where $H$ is symmetric. If the entries of $H$ are non-positive then $F$ is DR-submodular.

Another example is the \textit{generalized multilinear extension (GME)} of a submodular set function. Let $f:2^{V} \rightarrow \mathbb{R}_{+}$ be a set function over the ground set $V$. A GME can model problems involving \textit{decision-dependent uncertainty} in which resources are allocated to increase/decrease the probability that an element of the ground is set present. The function $f$ computes the utility of subsets of $V$ consisting of available elements.

Let $x\in \mathbb{R}^n$ be a vector of decision variables and let $g_i(x)$ be the probability that element $i\in V$ is available. Then the GME computes the expected utility as a function of $x$:

\begin{equation}
	\label{gme}
	F(x) = \mathbb{E}_{S\sim g(x)}[f(S)] = \sum_{S\subseteq V} f(S) \prod_{i\in S} g_i(x) \prod_{i\notin S} (1-g_i(x)),
\end{equation}

The case in which $g_i(x) = x_i$ is known as the \textit{multilinear extension}. In this case if $f$ is submodular then $F$ is a continuous DR-submodular function. For this special case $F$ has been used to obtain randomized algorithms for maximizing $f$ with a constant-factor approximation guarantee \citep{Calinescu.2011}.

Another special case is when allocating resources to element $i$ only affects the probability that $i$ itself is available, i.e., $g_i(x) = g_i(x_i)$. One example of this special case is the problem of allocating resources to facilities to increase the probability that facilities are available, with the ultimate goal of maximizing the expected utility. If $f$ is submodular and $g_i(x_i)$ is concave for all $i\in V$, then $F$ is DR-submodular \citep{Bian.2017}. See Sections \ref{subsec:fac-defense} and \ref{subsec:cap-fac-defense} for computational experiments with this application.

%An application of maximizing the general function \eqref{gme} is \textit{continuous influence maximization with general marketing strategies} \citep{kempe2003maximizing}. In this problem a decision-maker seeks to allocate resources among different marketing strategies in order to influence seed nodes in a social network with the goal of maximizing the expected number of nodes that are influenced through cascading influence that starts at the seed nodes. In this problem the decision vector describes the amount of resources allocated to different marketing strategies and $g_i(x)$ is the likelihood that a node $i$ is influenced and becomes and seed node given the marketing allocation $x$. See \S\ref{subsec:general-influence-max} for experiments on this application.

\subsection{Related Work}
Deterministic global optimization is a well-developed field of research in which the goal is to find a global optimum of a function without assuming global structure such as convexity. The different approaches for this class of problems can generally classified according to assumed problem structure.

Spatial-branch-and-bound approaches typically seek to form convex/concave underestimators/overestimators and then uses these to obtain lower/upper bounds within a branch-and-bound scheme. Estimators can be obtained for a factorable term by decomposing it into atomic units (e.g., $x$, $\sin(x)$, etc.) and operators $(\times,\div,+,-)$ linked together using an expression tree. Then the atomic units as well as product and quotient terms can be approximated above (below) using concave (convex) envelopes. For example, \cite{McCormick.1976} characterized the envelopes for the product of two functions, \cite{Meyer.2004v2} the envelopes for trilinear terms, and \cite{Rikun.1997} for multilinear terms. Because these approaches assume that an algebraic and factorable representation of the objective function and constraints is available, they can viewed as ``white-box'' optimization.

In stark contrast to methods that assume factorability, another area of research on \textit{derivative-free} or \textit{black-box} optimization uses only function evaluations to estimate the coefficients of a surrogate model without assuming any structure. Because these methods are designed for problems for which function evaluations are expensive, they typically do not use derivatives. Some approaches use interpolation (see \cite{Conn.2008}), while others use statistical models such as Gaussian processes \citep{Frazier.2018} or neural networks \citep{Snoek.2015}.

In between methods that assume factorability and derivative-free optimization is a small group of ``grey-box'' methods that do not require the function to be factorable but still use derivative/Hessian information. For example, the well-known $\alpha$-BB method (see \cite{Floudas.2000}) requires the function to be twice-differentiable and requires accurate estimation of a parameter $\alpha$ based on bounds of the eigenvalues of the Hessian. \cite{Meyer.2002} developed a method for problems with non-factorable constraints that uses a surrogate model and interpolation.

However, we are not aware of any ``grey-box'' continuous global optimization approaches that exploit submodularity. This paper investigates whether submodularity can be used to accelerate ``grey-box'' global optimization.

Submodularity is a much celebrated property in discrete optimization, partly because this property yields a constant factor approximation for many submodular maximization problems. For example, it is well known that non-decreasing set functions subject to matroid constraints can be approximated with guarantee $1-1/e$ \citep{Nemhauser.1978}. Further, it is well-known that submodular set functions can be minimized (without constraints) in polynomial time (see \cite{McCormick.2005}).

Several papers have used the properties of the hypograph of submodular set functions to obtain globally maximal solutions. \cite{Nemhauser.1981} presented an exact cutting plane algorithm for maximizing submodular set functions based on properties of their hypograph. \cite{Wu2018} used a version of this cutting plane algorithm to solve the discrete influence maximization problem, showing positive computational results. \cite{Ahmed.2011} study the convex hull of the hypograph of a special case of a submodular function consisting of a concave function composed with an additive function and present two new classes of valid inequalities that are stronger than the classic representation of \cite{Nemhauser.1981}. Others have used hypograph cuts as valid inequalities within a branch-and-cut algorithm for special classes of problems including variants of the facility location problem \citep{Ljubic.2018}. \cite{Yu.2021} present a cutting plane method for $k$-submodular function maximization.

General constrained submodular minimization is NP-hard \citep{Svitkina.2011}. \cite{Yu.2017j9x} study the convex hull of the epigraph set of a concave function composed with an additive function and provide a full description for the case in which the costs in the additive function are identical. They also provide a class of facet-defining inequalities for the non-identical. \cite{Yu.2020mvp} present a complete linear description of the convex hull of the epigraph of a bisubmodular function. \cite{Atamturk.2021} present inequalities that form an outer polyhedral approximation for the epigraph of submodular functions. \cite{Lee.2015} also contribute a cutting plane algorithm for continuous submodular minimization.

For the continuous case, is well known that submodular functions can be minimized in polynomial time \citep{Bach.2019}, Indeed, \cite{Bach.2019} showed that many of the results relating submodularity and convexity for set functions translate to continuous ones, focusing mostly on properties that are useful for minimization. 

Regarding continuous submodular maximization, it is known that continuous submodular functions can maximized with optimal guarantee $1-1/e$ in the non-decreasing case \citep{Bian.2017} and $1/2$ in the non-monotone case \citep{Niazadeh.2018}.

However, less is known about \textit{exact} submodular maximization in the continuous case. This paper focuses on addressing this gap in knowledge.

\subsection{Contributions and Findings}
Our specific contributions are the following:
\begin{enumerate}
    \item A characterization of the hypograph of continuous monotone DR-submodular functions via non-convex inequalities.
    
    \item A characterization of an approximate convex hull of the non-convex hypograph.
    
    %\item An empirical comparison of a cutting plane algorithm that optimizes over the concave envelope of the hypograph with a Frank-Wolfe algorithm.
    
    \item Computational results for a spatial branch-and-bound approach that uses the approximate convex hull of the hypograph to compute upper and lower bounds in each sub-rectangle.
    
\end{enumerate}

\subsection{Outline of the paper}
In the remainder of this article \S\ref{sec:background} provides technical details about submodular functions needed for the rest of the article as well as and overview of notation used. Section \ref{sec:characterize-hypograph} describes a set of inequalities that define the hypograph of a continuous non-decreasing DR-submodular function, and \S\ref{sec:exact-cutting-plane} describes a cutting plane procedure based on these inequalities. Because this exact cutting plane procedure is computationally intractable, \S\ref{sec:approx-hypograph} describes a set of inequalities that form the approximate convex hull of the hypograph of a continuous non-decreasing DR-submodular function, and \S\ref{sec:approx-cutting-plane-alg} describes a cutting plane procedure based on these inequalities. Section \ref{sec:spatial-bb-implementation} uses this approximate cutting plane algorithm as a subroutine within a spatial branch-and-bound algorithm, and \S\ref{sec:results} presents results of a computational study with this algorithm. Finally, Section \ref{sec:conclusion} concludes the article.

\section{Background and Notation} \label{sec:background}
In this section we provide several definitions and results that are needed for the remainder of the paper.

First, given a ground set $V$, a set function $f:2^{V} \rightarrow \mathbb{R}$ is \textit{submodular} iff:
\begin{equation}
    \label{submod-def}
    f(S) + f(T) \geq f(S \cup T)  + f(S \cap T) \quad \forall S, T \subseteq V
\end{equation}

Let $\rho_k(S) = f(S \cup \{k\}) - f(S)$ denote the marginal gain in $f$ from adding $k$ to the set $S$. A set function $f$ is non-decreasing if $\rho_k(S)\geq 0$ for all $S\subseteq V$. Using \eqref{submod-def}, it can be shown that $f$ is submodular if and only if the marginal gains are decreasing, i.e., 

\begin{equation}
    \rho_k(S) \geq \rho_k(T) \quad \forall S \subseteq T \subseteq V, k \in V \backslash T.
\end{equation}

\cite{Nemhauser.1978} showed that \eqref{submod-def} implies the following first-order overestimator property:

\begin{equation}
    \label{eqn:submod-set-overestimator}
    f(T) \leq f(S) + \sum_{j\in T\setminus S} \rho_j(S) - \sum_{j\in S\setminus T} \rho_j(S\cup T \setminus \{j\}) \quad \forall S,\,T\subseteq V.
\end{equation}

And when $f$ is non-decreasing they showed that \eqref{eqn:submod-set-overestimator} reduces to:

\begin{equation}
    \label{eqn:submod-set-overestimator-nondecr}
    f(T) \leq f(S) + \sum_{j\in T\setminus S} \rho_j(S)\quad \forall S,\,T\subseteq V.
\end{equation}

\cite{Nemhauser.1981} showed that \eqref{eqn:submod-set-overestimator-nondecr} define the hypograph of a non-decreasing submodular function and demonstrated how to use this result within a hypograph-based cutting plane algorithm. \cite{Wu2018} presented computational results for this cutting plane algorithm on an influence maximization problem.
				
\textit{Continuous} submodular functions are defined on subsets of the form $\mathbb{R}^{n}:\mathcal{X} = \prod_{i=1}^{n}\mathcal{X}_i$ where $\mathcal{X}_{i}$ is a compact subset of $\mathbb{R}$. Analogous to \eqref{submod-def}, a function $F:\mathcal{X} \rightarrow \mathbb{R}$ is submodular iff: 

\begin{equation} \label{eqn:cont-submod-defn}
    F(x) + F(y) \geq F(\max\{x,y\}) + F(\min \{x,y\}) \quad \forall (x,y) \in \mathcal{X} \times \mathcal{X}.
\end{equation}

Letting $e_{i}$ denote the $i$th basis vector, the function $\rho_{i}(\delta\,|\,x)=F(x+\delta e_{i}) - F(x)$ computes the marginal gain in $f$ from adding $\delta$ to component $x_i$. Thus, a continuous function $f$ is non-decreasing if $\rho_{i}(\delta\,|\,x)\geq 0$ for all $\delta \geq 0$ and $x\in \mathcal{X}$. Also, another necessary and sufficient for the submodularity of $F$ is 
\begin{equation}
    \rho_{i}(\delta\,|\,x) \geq \rho_{i}(\delta\,|\,x + \delta' e_{j}), \quad \forall x \in \mathcal{X}; \delta,\delta'\in \mathbb{R}^+.
\end{equation}

Hereafter we assume that $F$ is DR-submodular, meaning that it is submodular and also component-wise concave \citep{Bian.2017}.

In this paper we use the following notational conventions. First, $[x]^+$, denotes the rectified linear unit function $\max\{x,0\}$ applied to the vector $x$ elementwise so that the $i$th element is $\max\{x_i,0\}$. The notation $[n]$ defines the set $\{1,\dots,n\}$. Second, when $x$ and $y$ are vectors, then $x\odot y$ and $x/y$ denote elementwise addition and division, respectively. Finally, when $\ell$ and $u$ are scalars the notation $[\ell,u]$ denotes the interval from $\ell$ to $u$. However, when $\ell$ and $u$ are vectors then $[\ell,u]$ denotes the rectangle $\{x\,:\, \ell_i \leq x_i \leq u_i \, \forall i\in [n]\}$, where $n$ is the dimension of the space.

\section{Characterization of the hypograph of \normalfont \emph{F}}\label{sec:characterize-hypograph}
Just as \cite{Nemhauser.1981} showed that the first-order overestimators \eqref{eqn:submod-set-overestimator-nondecr} define the hypograph of non-decreasing submodular set functions, in this section we seek to develop an equivalent result for the continuous case.

\subsection{Results for general continuous \normalfont \emph{F}}
First, we show that an analog to the first-order estimator for submodular set functions \eqref{eqn:submod-set-overestimator} exists for continuous functions.

\begin{lemma} \thlabel{res:inequality-general}
For a continuous submodular function $F$, the following holds for all $x,y \in \mathcal{X}$:
\begin{equation}
       F(y) \leq F(x) + \sum_{t=1}^n \rho([y_t-x_t]^+\,|\, x) - \sum_{t=1}^n \rho([x_t-y_t]^+ e_t\,|\, \max\{x,y\} - \max\{x_t-y_t\} e_t ) \label{eqn:first-order-overestimator-cont}
\end{equation}
\end{lemma}

\begin{proof}
We employ a strategy that is similar to the one used by \cite{Nemhauser.1978} for the set function case. Let $z=\max(y-x,0)$ and $w=\max(x-y,0)$. Then,

\begin{equation} \label{lastinequality-1}
\begin{aligned} 
F(\max\{x,y\}) - F(x) &=  \sum_{t=1}^n \left[ F\left(x + \sum_{i=1}^t z_i e_i \right) - F\left(x + \sum_{i=1}^{t-1} z_i e_i \right) \right] \\ 
&=  F(x+z_1 e_1) - F(x) + F(x+z_1 e_1 + z_2 e_2) - F(x+z_1 e_1) + \dots \\
&= \sum_{t=1}^n \rho\left(z_t\,\bigg\rvert\, x + \sum_{i=1}^t z_i e_i\right) \\
&\leq \sum_{t=1}^n \rho(z_t\,|\, x), 
\end{aligned}
\end{equation}

where the last inequality is due to the submodularity of $F$.

\begin{equation} \label{lastinequality-2}
\begin{aligned} 
F(\max\{x,y\}) - F(y) &=  \sum_{t=1}^n \left[ F\left(y + \sum_{i=1}^t w_i e_i\right) - F\left(y + \sum_{i=1}^{t-1} w_i e_i \right) \right] \\ 
&=  \sum_{t=1}^n \rho\left(w_t e_t\,\bigg\rvert\, y + \sum_{i=1}^t w_i e_i  - w_t e_t\right)\\
&\geq \sum_{t=1}^n \rho(w_t e_t\,|\, \max\{y,x\} - w_t e_t),
\end{aligned}
\end{equation}

Using \eqref{eqn:cont-submod-defn}, the result is obtained by subtracting \eqref{lastinequality-2} from \eqref{lastinequality-1}.

\end{proof}

Next, we show that when $F$ is non-decreasing, \eqref{eqn:first-order-overestimator-cont} simplifies to the following first-order estimator. This estimator is analogous to the result for submodular set functions \eqref{eqn:submod-set-overestimator-nondecr}.

\begin{proposition} \thlabel{res:inequality-non-decreasing}
For a non-decreasing continuous DR-submodular function $F$, the following holds:
   \begin{equation} \label{eqn:first-order-overestimator-cont-nondecr}
       F(x) \leq F(\hat{x}) + \sum_{t=1}^n \rho([x_t-\hat{x}_t]^+\,|\, \hat{x}) \quad \forall x, \hat{x} \in \mathcal{X}
   \end{equation}
\end{proposition}

\begin{proof}
For a non-decreasing $F$, $\rho_i(\delta\,|\, x)\geq 0$ for all $i$, $\delta\geq 0$, and $x\in \mathcal{X}$. Thus, the last term of \eqref{lastinequality-1} is non-positive, yielding the result.

\end{proof}

\subsection{Results for differentiable \normalfont \emph{F}}

Next, we consider the following gradient-based first-order estimator of $F$ based on support point $\hat{x}$ (see Figure \ref{fig:first-order-overestimators}):

\begin{equation} \label{eqn:gradient-first-order-estimator}
    \tilde{F}(x;\hat{x}) = F(\hat{x}) + \nabla F(\hat{x})^{T}[x-\hat{x}]^{+} \quad \forall x,\hat{x}\in \mathcal{X},
\end{equation}

as well as the overestimator $\tilde{F}(x) = \min_{\hat{x}\in \mathcal{X}} \tilde{F}(x;\hat{x})$ defined over $\mathcal{X}$. The following result shows that $\tilde{F}(x)$ is a perfect estimator and that $\tilde{F}(x;\hat{x})$ is a global overestimator of $F(x)$.

\begin{proposition} \thlabel{res:gradient-based-first-order}
(Gradient-based first-order overestimator.) For a non-decreasing continuous differentiable DR-submodular function $F$, the following hold:
\begin{equation} \label{eqn:gradient-first-order-overestimator}
   F(x) \leq \tilde{F}(x;\hat{x}) \quad \forall x, \hat{x} \in \mathcal{X}
\end{equation}
\begin{equation} \label{eqn:perfect-estimator-at-xHat}
   \tilde{F}(x) = F(x) \quad \forall x \in \mathcal{X}
\end{equation} 
\end{proposition}

\begin{proof}
To show \eqref{eqn:gradient-first-order-overestimator}, 
because DR-submodular functions are  component-wise concave, $\rho_{i}([x_{i}-\hat{x}_{i}]^{+}\,|\,x) \leq \frac{\partial F(\hat{x})}{\partial x_{i}}[x_{i}-\hat{x}_{i}]^{+}$. Thus, because $\tilde{F}(x;\hat{x})$ is an upper bound on the right hand side of \eqref{eqn:first-order-overestimator-cont-nondecr}, the result holds. Then, \eqref{eqn:perfect-estimator-at-xHat} follows from \eqref{eqn:gradient-first-order-overestimator} and the fact that $\tilde{F}(x;x) = F(x)$ (see \eqref{eqn:gradient-first-order-estimator}). 
\end{proof}

Figure \ref{fig:first-order-overestimator-1d} shows the gradient-based first-order overestimator for the function $F(x)=\sqrt{x}$, and \ref{fig:first-order-overestimator-2d} for a two-dimensional submodular function. Note that these overestimators are piecewise linear but not concave.

\begin{figure}
    \begin{subfigure}{0.45\textwidth}
	\centering
	\includegraphics[width=\textwidth]{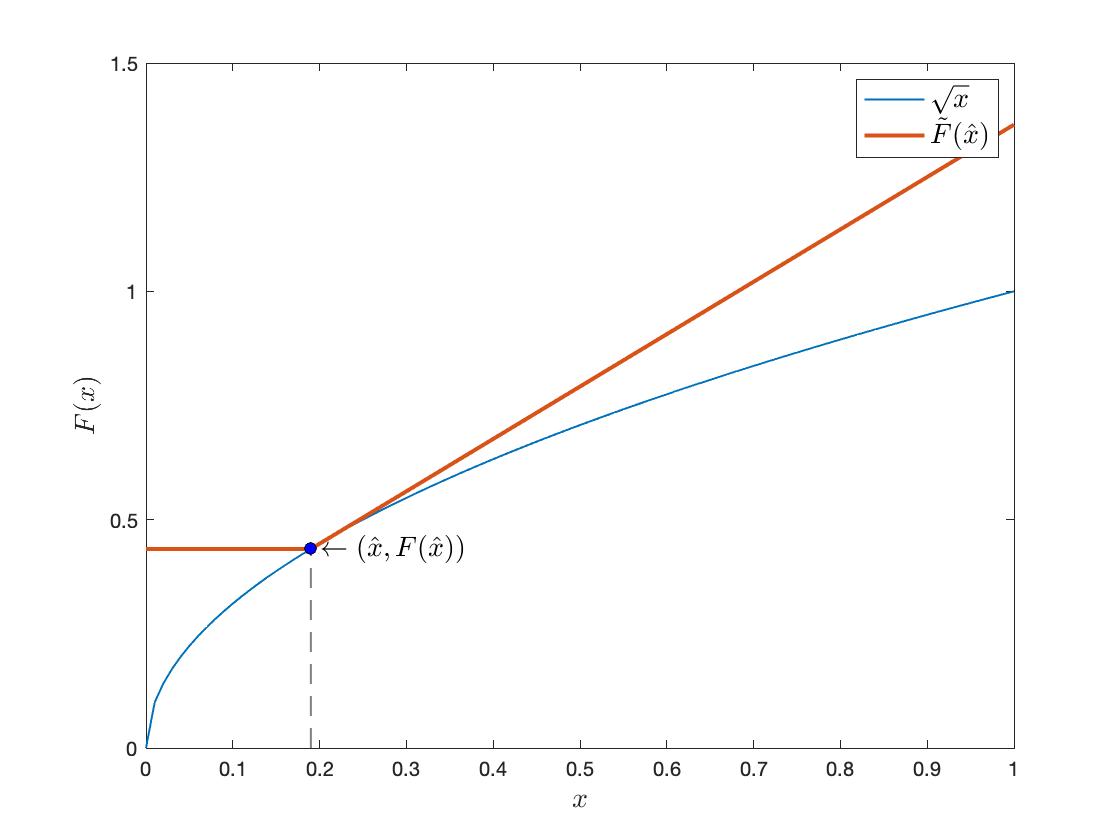}
	\caption{Gradient-based first-order overestimator $\tilde{F}(x;\hat{x})$ of the function $F(x)=\sqrt{x}$.}
	\label{fig:first-order-overestimator-1d}
	\end{subfigure}
	\hspace{1em}% Space between image A and 
	\begin{subfigure}{0.5\textwidth}
	\centering
	\includegraphics[width=\textwidth]{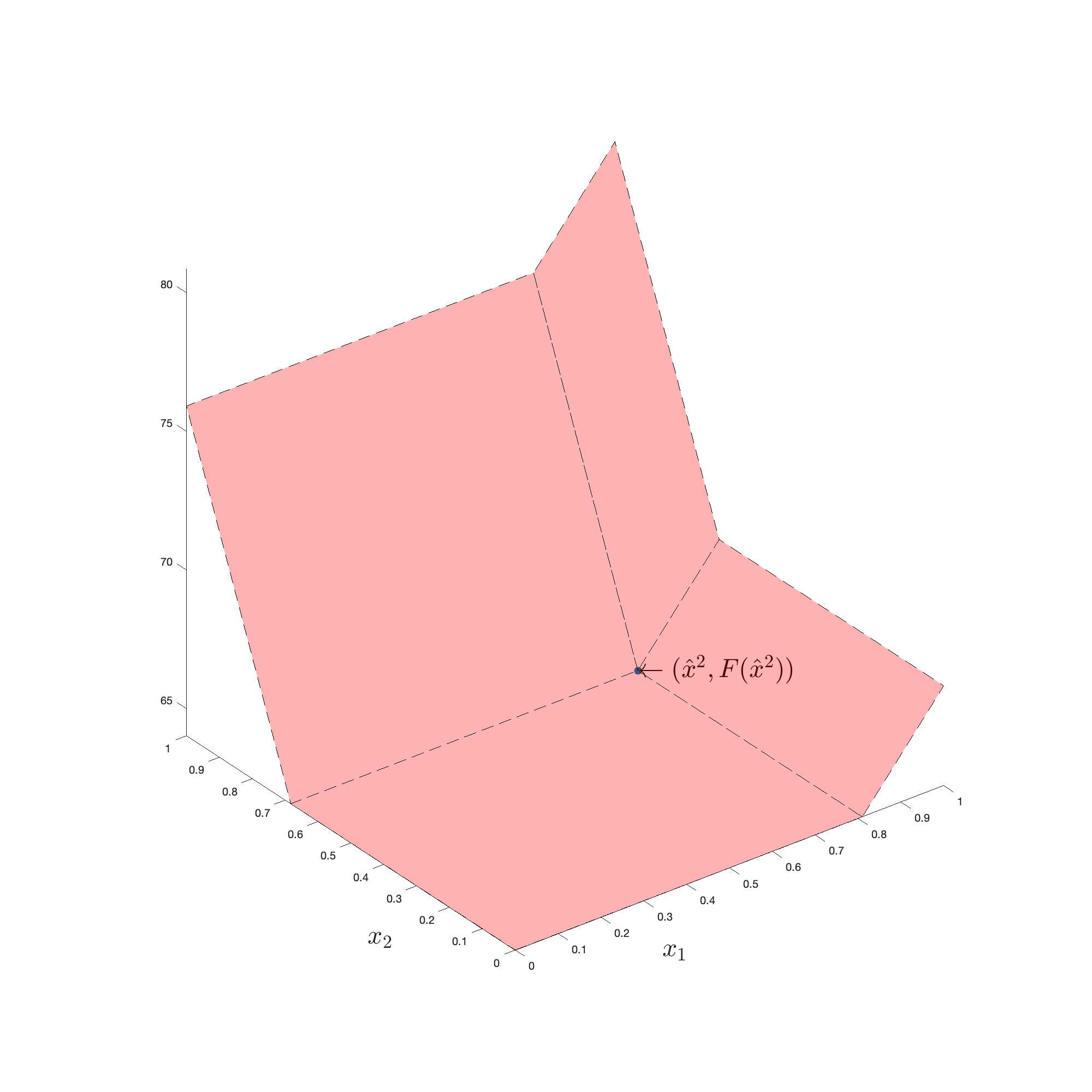}
	\caption{Gradient-based first-order overestimator $\tilde{F}(x;\hat{x})$ of 2-dimensional submodular function.}
	\label{fig:first-order-overestimator-2d}
    \end{subfigure}
    \caption{Gradient-based first-order overestimators $\tilde{F}(x;\hat{x})$ \eqref{eqn:first-order-overestimator-cont-nondecr}.}
    \label{fig:first-order-overestimators}
\end{figure}

Next, we show that the gradient-based first-order overestimators \eqref{eqn:gradient-first-order-overestimator} form the hypograph of $F$. Let $\text{hyp } F = \{(\eta, x) \,|\, \eta \leq F(x) \}$ denote the hypograph of a continuous submodular function $F$. Further, let $Y =\{(\eta,x) \,|\, \eta \leq \tilde{F}(x) \text{ for all } x \in \mathcal{X}\}$.

\begin{lemma} \thlabel{res:equivalent-sets}
For a differentiable non-decreasing continuous DR-submodular function $F$, $(\eta,x)\in Y$ iff 
$(\eta,x) \in \textrm{hyp } F$.
\end{lemma}

\begin{proof}

Suppose $(\eta,x)\in \textrm{hyp } F$. Then for all $\hat{x}\in \mathcal{X}$
\[ F(\hat{x}) + \nabla F(\hat{x})^{T}[x-\hat{x}]^{+} \geq F(x) \geq \eta, \]

where the first inequality follows from \thref{res:gradient-based-first-order}. Conversely, $(\eta,x)\in Y$ implies that

\[\eta \leq \tilde{F}(x;x) = F(x) + \nabla F(x)^{T}[x-x]^{+} = F(x)\]

\end{proof}

\section{Hypograph Reformulation and Cutting Plane Algorithm} \label{sec:exact-cutting-plane}
Given that the gradient-based first-order estimators form the hypograph of $F$ (\thref{res:equivalent-sets}), in this section we investigate the following hypograph-based reformulation of our original maximization problem \eqref{opt-problem}

\begin{equation}
	\label{eqn:hypograph-formulation}
	\max_{(\eta,x)\in \text{ hyp } F, x\in \mathcal{X}} \eta.
\end{equation}

\subsection{Hypograph Reformulation}
Given the fact that \eqref{eqn:gradient-first-order-estimator} define the hypograph of $F$ (see \thref{res:equivalent-sets}), we can reformulate \eqref{eqn:hypograph-formulation} as:

\begin{subequations}
\begin{alignat}{3}
 & \max_{x\in \mathcal{X}} & \eta & \\
 & \textrm{subject to} \quad & \eta & \leq F(\hat{x}) + \nabla F(\hat{x})^{T}[x-\hat{x}]^+ \quad && \forall \hat{x}\in \mathcal{X}. \label{constr:hypograph-reform-cut}
\end{alignat}
\label{exact-cp-main}
\end{subequations}

Problem \eqref{exact-cp-main} is equivalent to \eqref{opt-problem} in the sense that they have the same optimal objective value. This is because \thref{res:equivalent-sets} implies:

\begin{corollary} \thlabel{res:cutting-plane-equivalence}
(Cutting plane formulation.) $(\eta,x)$ is an optimal solution to \eqref{exact-cp-main} iff $x$ is an optimal solution to \eqref{opt-problem}.
\end{corollary}

Because constraints \eqref{constr:hypograph-reform-cut} are non-convex, we will now present an extended linear formulation with binary variables. %Specifically, we will present an a formulation with the property of being \textit{ideal}, meaning that the extreme points of the formulation are binary.

%First, we will introduce some notation. Let $S = \text{gr}(g;D)$ be a set that is equal to the graph of function $g$.
%A mixed-integer programming (MIP) formulation of $S$ consists of linear constraints on $(x,y,z)\in\mathbb{R}^{n+1}$ that define a polyhedron $Q$, integrality constraints $z\in\{0,1\}^n$, such that

%\[S = \text{Proj}_{x,y}(Q\cap (\mathbb{R}^{n+1} \times \{0,1\}^n)),\]

%where $\text{Proj}_{x,y}(R)=\{(x,y)\,|\, \exists z \text{ such that } (x,y,z)\in R\}$ is the orthogonal projection of set $R$ onto variables $(x,y)$. Let $\text{ext}(Q)$ the set of extreme points of a polyhedron $Q$. A MIP formulation is \textit{ideal} if $\text{ext}(Q)\subseteq \mathbb{R}^{n+1}\times \{0,1\}$.

We first show an \textit{ideal} formulation for the \textit{graph} of the positive part of of a particular variable in \eqref{constr:hypograph-reform-cut}, i.e., $[x_i-\hat{x}_i]^+$. An ideal formulation is one for which the extreme points of the linear programming relaxation are binary-valued.

This graph over the interval $[\ell_i,u_i]$ is defined as:

\begin{equation}
    \text{gr}([x_i-\hat{x}_i]^+;[\ell_i,u_i]) = \{(x_i,y_{i}) \in \mathbb{R}^n \times \mathbb{R}, y_{i} = [x_i-\hat{x}_i]^+,\ell_i\leq x_i\leq u_i\}
\end{equation}

\cite{Anderson.2018} show that a tightened big-$M$ formulation is an ideal formulation for $\text{gr}([w^Tx + b]^+;[\ell,u])$. A special case of this result is that the following formulation is ideal for $\text{gr}([x_i-\hat{x}_i]^+;[\ell_i,u_i])$:

\begin{subequations}
\begin{alignat}{3}
    && y_{i}& \leq x_i - \ell_i(1 - z_{i})- \hat{x}_i z_{i} ,\label{eqn:exact-single-var-linearize1}\\
    && y_{i}& \leq (-\hat{x}_i + u_i)z_{i},\label{eqn:exact-single-var-linearize2}\\
    && y_{i}& \geq x_i - \hat{x}_i,\label{eqn:exact-single-var-linearize3}\\
    && (x_i,y_{i},z_i) & \in \mathbb{R}\times \mathbb{R}_{\geq 0} \times [0,1], \label{eqn:exact-single-var-linearize4} \\
    && z_{i}& \in \{0,1\}. \label{eqn:exact-single-var-linearize5}
\end{alignat}
\label{exact-main-prob-bigm-single-var}
\end{subequations}

Given that each constraint \eqref{constr:hypograph-reform-cut} is a weighted sum of separable $[x_i-\hat{x}_i]^+$ terms plus a constant, then the ideal formulations for each $[x_i-\hat{x}_i]^+$ term can be combined to form an ideal formulation of the graph of $F(\hat{x}) + \nabla F(\hat{x})^{T}[x-\hat{x}]^+$ for a given $\hat{x}$. 

%Given Lemma \ref{lem:hypograph}, the following is an (infinite) ideal formulation 
The following is a valid, albeit infinite, formulation for the hypograph of $F$.

\begin{subequations}
\begin{alignat}{3}
    && y_{i,\hat{x}}& \leq x_i - \ell_i(1 - z_{i,\hat{x}})- \hat{x}_i z_{i,\hat{x}} & \quad\forall i\in [n],\hat{x}\in \mathcal{X}, \label{eqn:hypograph-form-1}\\
    && y_{i,\hat{x}}& \leq (-\hat{x}_i + u_i)z_{i,\hat{x}} & \forall i\in [n],\hat{x}\in \mathcal{X},\label{eqn:hypograph-form-2}\\
    && y_{i,\hat{x}}& \geq x_i - \hat{x}_i & \forall i\in [n],\hat{x}\in \mathcal{X},\label{eqn:hypograph-form-3}\\
    && x & \in [\ell,u], \label{eqn:hypograph-form-4} \\
    && z_{i,\hat{x}}& \in \{0,1\} & \forall i\in [n],\hat{x}\in \mathcal{X},\\
    && y_{i,\hat{x}}& \geq 0 & \forall i\in [n],\hat{x}\in \mathcal{X},\\
    && \eta & \leq F(\hat{x}) + \sum_{i=1}^n \frac{\partial F(\hat{x})}{x_i}  y_{i,\hat{x}} & \forall \hat{x}\in \mathcal{X},\\
    &&  \eta & \geq 0.
\end{alignat}
\label{exact-main-prob-bigm}
\end{subequations}

\subsection{Hypograph-Based Cutting Plane Algorithm}
Given the result of \thref{res:cutting-plane-equivalence} and the fact that \eqref{exact-main-prob-bigm} is an infinite formulation, we present the following exact cutting plane algorithm. Let $Q$ be the number of iterations executed by the algorithm thus far, producing support points $\hat{x}^1, \dots, \hat{x}^Q$. The cutting plane main problem after $Q$ iterations is:

\begin{subequations}
\begin{alignat}{3}
    & \max_{x\in \mathcal{X}} & \eta & \\
    & \textrm{subject to} \quad & y_{iq}& \leq x_i - \ell_i(1 - z_{iq})- \hat{x}_i^q z_{iq} & \quad \forall i\in [n],q\in[Q], \label{eqn:cutting-plane-constr-1}\\
    && y_{iq}& \leq (-\hat{x}_i^q + u_i)z_{iq} & \forall i\in [n],q\in[Q],\label{eqn:cutting-plane-constr-2}\\
    && y_{i,q}& \geq x_i - \hat{x}_i^q & \forall i\in [n],q\in[Q],\label{eqn:hypograph-form-3-master}\\
    && z_{i,q}& \in \{0,1\} & \forall i\in [n],q\in[Q],\label{eqn:hypograph-form-4-master}\\
    && y_{i,q}& \geq 0 & \forall i\in [n],q\in[Q],\label{eqn:hypograph-form-5-master}\\
    && \eta & \leq F(\hat{x}^q) + \sum_{i=1}^n \frac{\partial F(\hat{x}^q)}{x_i}  y_{iq} & q\in[Q],\\
    &&  \eta & \geq 0. \label{eqn:cutting-plane-constr-last}
\end{alignat}
\label{formulation:exact-restricted-main}
\end{subequations}

%However, despite the fact that constraints \eqref{eqn:cutting-plane-constr-1}--\eqref{eqn:cutting-plane-constr-last} form an ideal formulation, the complete formulation that includes the constrain $x\in \mathcal{X}$ may no longer be ideal.

Algorithm \ref{alg:exact-cutting-plane} shows pseudocode for the exact cutting plane algorithm.

\begin{algorithm}
	{\sc ExactCuttingPlane}($\ell, u, \epsilon$)
	\begin{algorithmic}[1]
		\STATE $q \gets 1$
		\STATE $LB_q \gets 0, UB_q \gets +\infty$
		\WHILE{$UB_q - LB_q > \epsilon$}
		\STATE Solve main problem \eqref{formulation:exact-restricted-main} to obtain solution $(\eta^*,x^*)$.
		\STATE $\hat{x}^q \gets x^*$
		\STATE Add a cut to the main problem of the form: $\eta \leq F(\hat{x}^q) + \sum_{i=1}^n \frac{\partial F(\hat{x}^q)}{x_i}  y_{iq}$ along with variables $y_{iq}$ and $z_{iq}$ and constraints \eqref{eqn:cutting-plane-constr-1}--\eqref{eqn:hypograph-form-5-master} for $i\in[n]$. \label{alg:add-cut-step}
		\STATE $LB_q \gets F(x^*), UB_q \gets \eta^*$.
		\STATE $q \gets q+1$.
		\ENDWHILE
		\STATE \RETURN $LB_q$, $UB_q, x^*$
	\end{algorithmic}
	\caption{Exact cutting plane algorithm for continuous submodular maximization.}
	\label{alg:exact-cutting-plane}
\end{algorithm}

The following result shows that Algorithm \ref{alg:exact-cutting-plane} outputs an $\epsilon$-optimal solution.

\begin{proposition} \thlabel{cutting-plan-equivalence}
(Convergence of exact cutting plane algorithm.) $F(x^*) + \epsilon \geq \max_{x\in \mathcal{X}} F(x)$, where $x^*$ is the solution returned by Algorithm \ref{alg:exact-cutting-plane}.
\end{proposition} 

\begin{proof}
The algorithm terminates at iteration with a solution $(\eta^*, x^*)$ such that $\eta^* - F(x^*) \leq $. Because $\eta^* \geq \max_{x\in \mathcal{X}}$ the result is obtained. 
\end{proof}

Preliminary experiments showed that the exact cutting plane algorithm (Alg. \ref{alg:exact-cutting-plane}) required significant computation time, in part because the main problem at iteration $Q$ is a mixed-integer program with $nQ$ binary variables. Thus, in the next section we investigate an approximation of the convex hull of the hypograph of a submodular function that yields a cutting plane algorithm with a main problem that is a linear program.

\section{Characterization of the approximate convex hull of the hypograph of \normalfont \emph{F}} \label{sec:approx-hypograph}

In this section we seek to find the convex hull the hypograph of a submodular function $F$ and then use a cutting plane algorithm to optimize over this convex hull. 

\subsection{Convex hull of individual cuts}
We begin by finding the convex hull of the hypograph of the gradient-based first-estimator $\tilde{F}(x;\hat{x})$. 

We start by applying the well-known $\Delta$ or triangle envelope of the univariate ReLU function $[x]^+$ to $[x-\hat{x}]^+$ on the interval $[\ell,u]\subseteq \mathbb{R}$; that is:

\begin{equation} \label{eqn:relu-envelope}
    \breve{F}(x;\hat{x}) = F'(\hat{x})\frac{u-\hat{x}}{u-\ell}(x-\ell).
\end{equation}

It is well-known that this is the concave envelope of $[x-\hat{x}]^+$, as Figure \ref{fig:approx-cuts-1d} illustrates.

Because the the concave envelope operation is closed under summation \citep{Falk.1969}, then the following forms the concave envelope of the first-order estimator $\tilde{F}(x;\hat{x})$ in the multivariate case.

\begin{equation}
    \breve{F}(x;\hat{x}) = F(\hat{x}) + \left(\nabla F(\hat{x})\odot \frac{u-\hat{x}}{u - \ell} \right)^T (x-\ell)
    \label{eqn:approx-first-order-estimator}
\end{equation}

Figure \ref{fig:approx-cuts-2d} illustrates $\breve{F}(x;\hat{x})$ for a two-dimensional submodular function.

\begin{figure}
    \begin{subfigure}{0.45\textwidth}
	\centering
	\includegraphics[width=\textwidth]{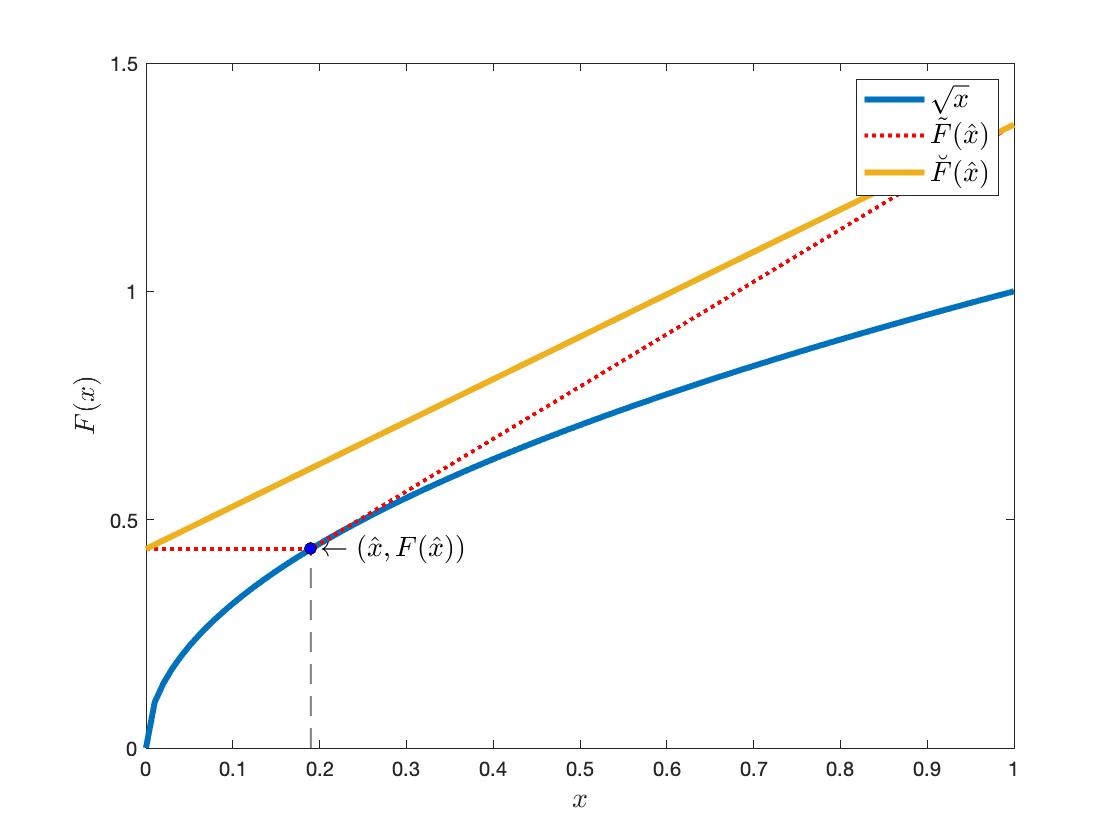}
	\caption{$F(x)=\sqrt{x}$ in blue, gradient-based first-order estimator $\tilde{F}(x;\hat{x})$ as the red dashed line, and the approximate first-order estimator $\breve{F}(x;\hat{x})$ in yellow.}
	\label{fig:approx-cuts-1d}
	\end{subfigure}
	\hspace{1em}% Space between image A and 
	\begin{subfigure}{0.5\textwidth}
	\centering
	\includegraphics[width=\textwidth]{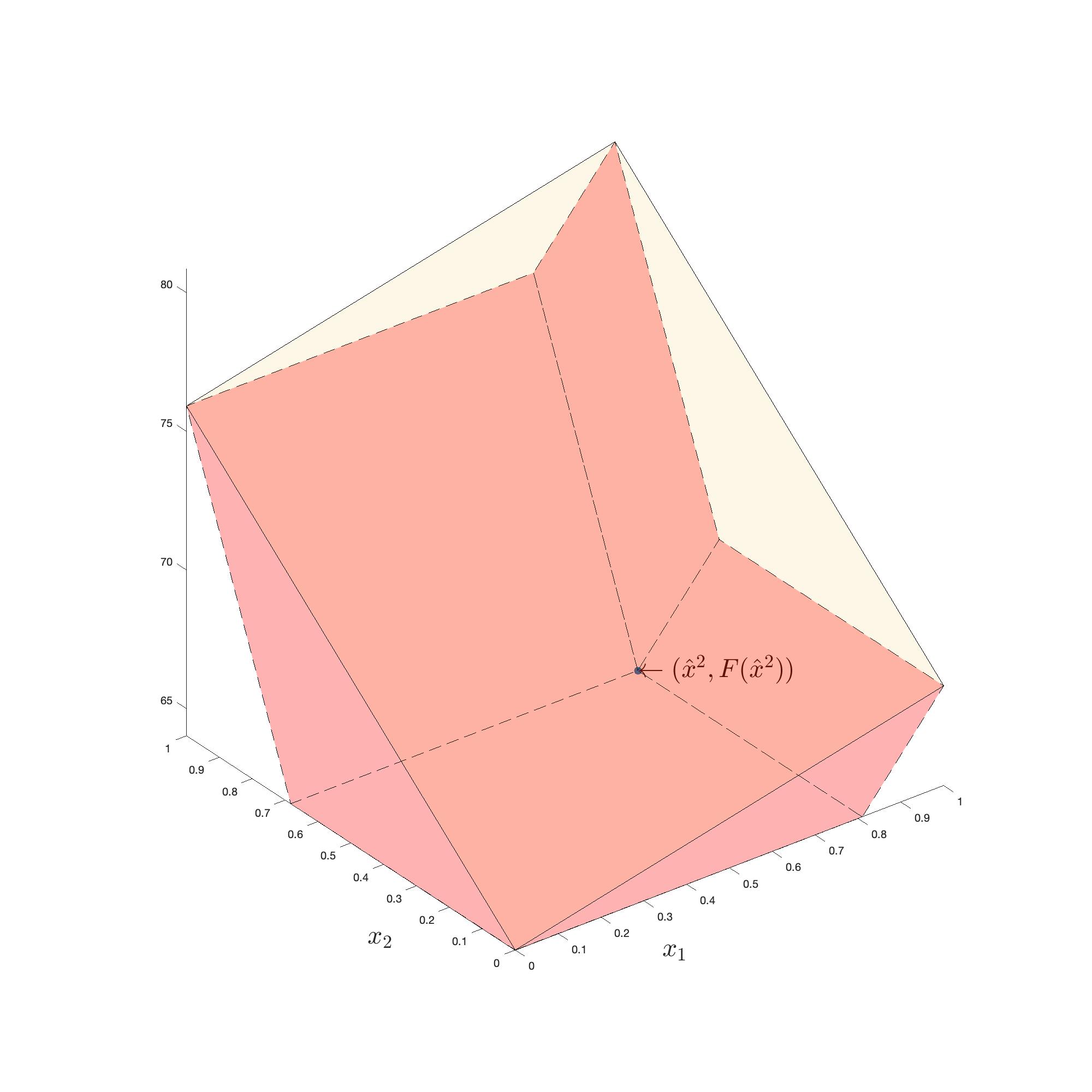}
	\caption{Gradient-based first-order estimator $\tilde{F}(x;\hat{x}^2)$ of 2-dimensional submodular function (obtained after 2nd iteration in cutting plane algorithm) in red and the approximate first-order estimator $\breve{F}(x;\hat{x}^2)$ in yellow.}
	\label{fig:approx-cuts-2d}
    \end{subfigure}
    \caption{Approximate cuts for two functions.}
    \label{fig:approx-cuts}
\end{figure}

It would be convenient if the minimum of the concave envelopes, i.e., $\min_{\hat{x}\in X}\breve{F}(x;\hat{x})$, formed the concave envelope of the minimum of the first-order estimators, i.e., $\min_{\hat{x}\in X} \tilde{F}(x;\hat{x})$. Unfortunately, this is may not be the case, as Figures \ref{fig:approxEnvVsConcaveEnv-1D} and \ref{fig:approxEnvVsConcaveEnv-2D} illustrate.

\begin{figure}
	\centering
	\includegraphics[width=\linewidth]{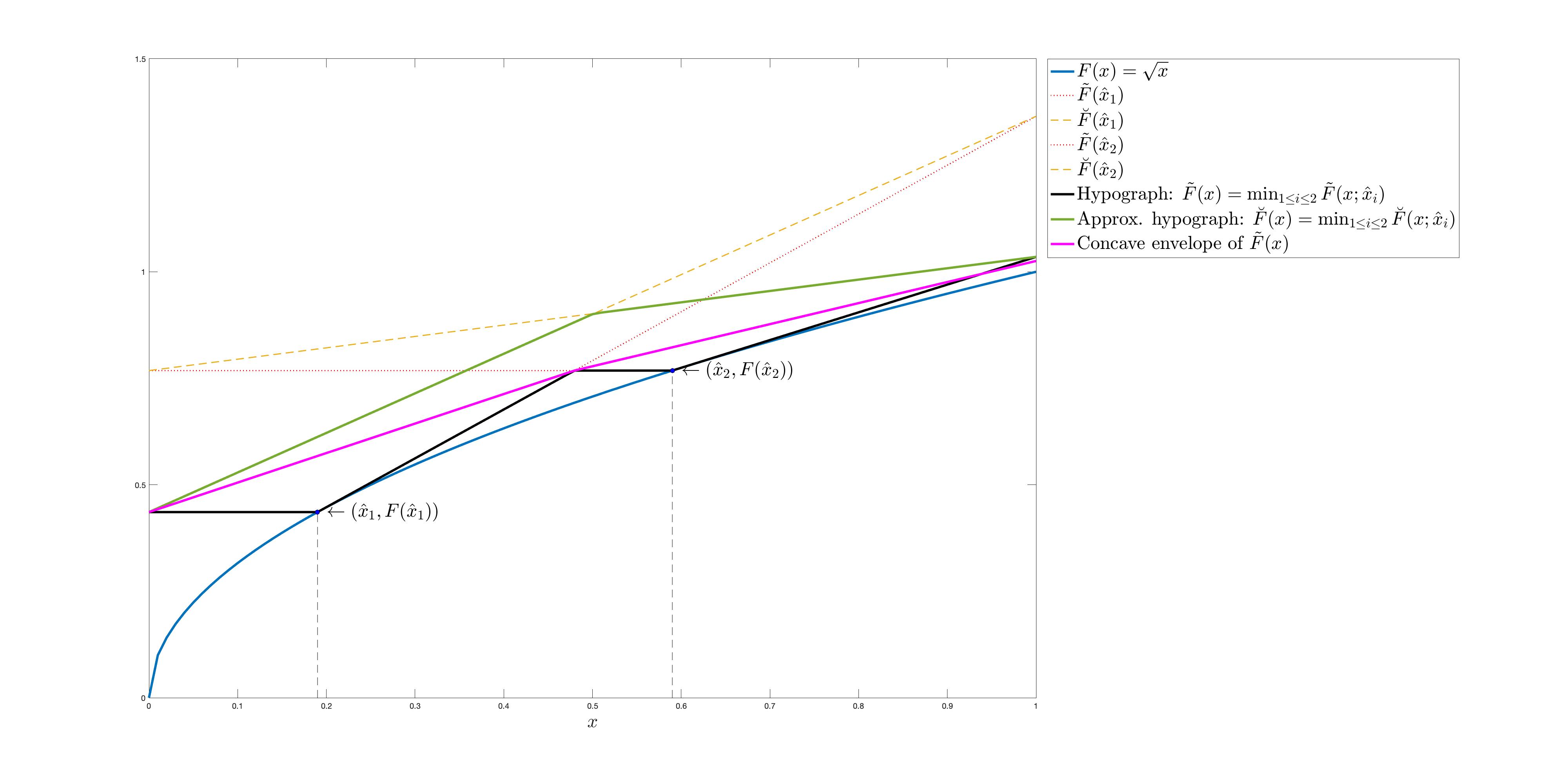}
	\caption{Approximate concave envelope (green) versus exact concave envelope (magenta) for univariate function $F(x)=\sqrt{x}$. As the figure illustrates, the exact concave envelope is below the approximate envelope.}
	\label{fig:approxEnvVsConcaveEnv-1D}
\end{figure}

In one dimension the concave envelope of $\min_{1\leq q\leq Q} \tilde{F}(x;\hat{x}^q)$ is a piecewise linear function with $Q+1$ pieces and can be modeled with disjunctions, e.g., using the well-known multiple-choice formulation (see \cite{Vielma.2010}). However, in general the minimum of several cuts (i.e., $\min_{1\leq q\leq Q} \tilde{F}(x;\hat{x}^q)$) has a large number of faces, which makes finding the concave envelope difficult (Figure \ref{fig:approxEnvVsConcaveEnv-2D} shows the polyhedral surface for the minimum of three cuts in two dimensions). Thus, in this paper we use the minimum of approximate cuts, i.e., $\min_{1\leq q\leq Q} \breve{F}(x;\hat{x}^q)$ to approximate the hypograph of $F$.

\begin{figure}
	\centering
	\includegraphics[width=\linewidth]{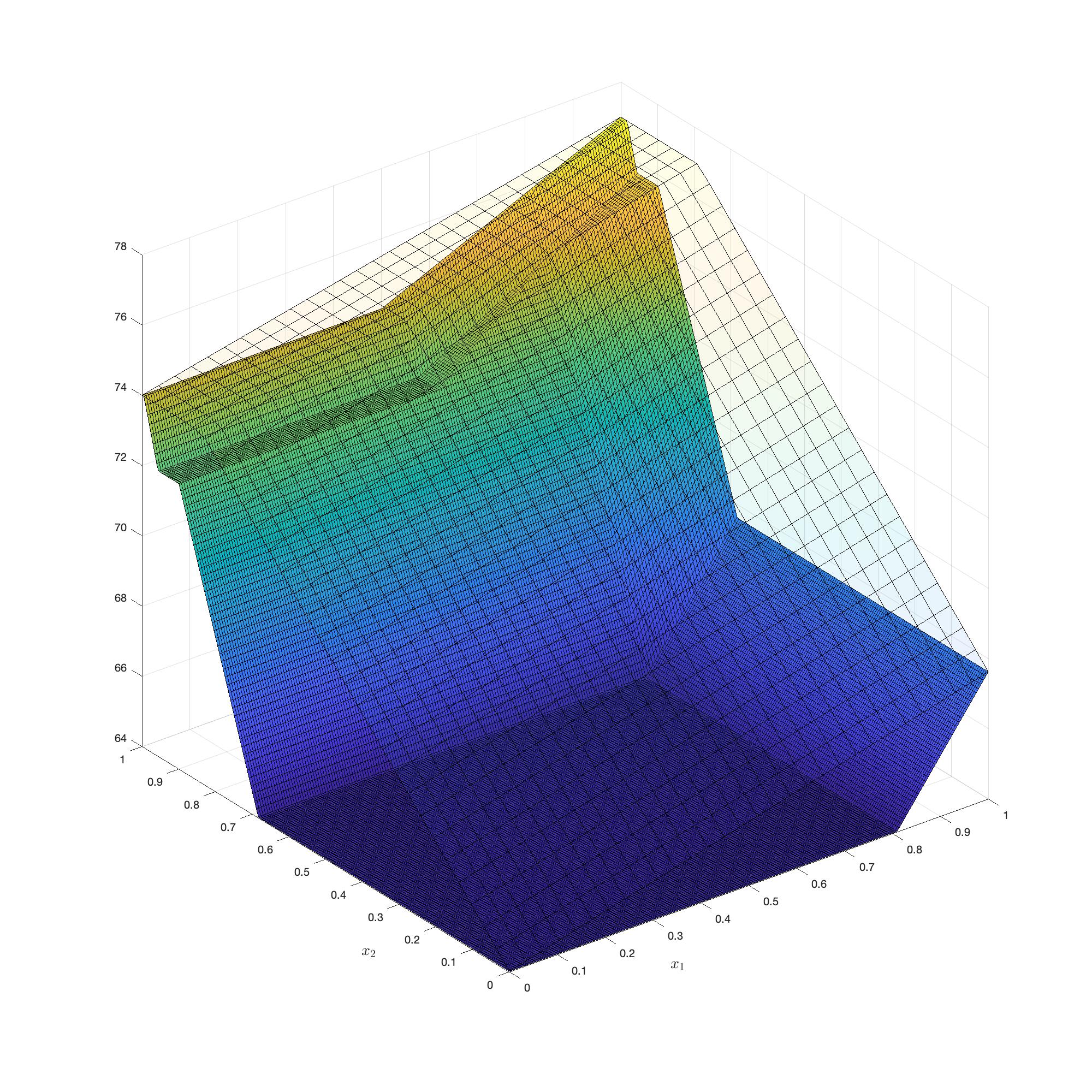}
	\caption{The lower surface is the minimum of three cuts (i.e., $\tilde{F}(x) = \min_{1\leq q\leq 3} \tilde{F}(x;\hat{x}^q)$) for a two-dimensional continuous submodular function. The upper surface is the approximate concave envelope $\breve{F}(x)=\min_{1\leq q\leq 3} \breve{F}(x;\hat{x}^q)$. As the figure shows, $\breve{F}(x)$ is not the exact concave envelope of $\tilde{F}(x)$.}
	\label{fig:approxEnvVsConcaveEnv-2D}
\end{figure}

\subsection{Approximation error of approximate convex hull}
In this section we investigate the worst-case error incurred when approximating the hypograph of a submodular function $F$ via its approximate convex hull.

First, we analyze the difference between the approximate estimators and the non-concave first-order estimators, i.e., $\breve{F}(x;\hat{x}) - \tilde{F}(x;\hat{x})$.

The error for a single cut is:

\begin{equation}
\breve{F}(x;\hat{x}) - \tilde{F}(x;\hat{x}) \leq \nabla F(\hat{x})^T\left( \frac{u-\hat{x}}{u - \ell}\odot (x-\ell) - [x-\hat{x}]^+ \right)
    \label{eqn:error-bound1}
\end{equation}

The following result identifies the point at which the error is maximized.

\begin{proposition} \label{res:error-maximizer}
For a given $\hat{x}\in [\ell,u]$, the error is maximized at the point $x=\hat{x}$ so that the following holds.

\begin{equation}
    \breve{F}(x;\hat{x}) - \tilde{F}(x;\hat{x}) \leq  \nabla  F(\hat{x}) ^T \left( \frac{u-\hat{x}}{u-\ell} \odot (\hat{x}-\ell) \right)
    \label{eqn:error-bound2}
\end{equation}

\end{proposition}

\begin{proof}
Because the right-hand side of \eqref{eqn:error-bound1} is separable, it is enough to find the maximizer of $\frac{\partial F(\hat{x})}{\partial x_i} \left(\frac{u_i-\hat{x}_i}{u_i-\ell_i}(x_i - \ell_i)-[x_i-\hat{x}_i]^+\right)$. We divide the proof into two overlapping cases. First, when $x_i\leq \hat{x}_i$ the error is maximized at the upper limit $\hat{x}_i$ because the derivative of the error with respect to $x_i$ is $\frac{\partial F(\hat{x})}{\partial x_i} \frac{u_i-\hat{x}_i}{u_i-\ell_i}$, which is non-negative due to $F$ being non-decreasing. For the second case in which $x_i \geq \hat{x}_i$ this derivative is $\frac{\partial F(\hat{x})}{\partial x_i} \left( \frac{u_i-\hat{x}_i}{u_i-\ell_i}-1\right)\leq 0$. Thus, the error is maximized at the lower limit $\hat{x}_i$.
\end{proof}

The following result (without proof) lists two more implications of \eqref{eqn:error-bound1}.

\begin{corollary} \thlabel{res:implications-of-max-error}
For a differentiable non-decreasing continuous DR-submodular function $F(x)$, the following hold:
\begin{enumerate}
    \item For $x \in[\ell,u]\subseteq \mathbb{R}$, $\breve{F}(x;u) = \tilde{F}(x;u)$ and $\breve{F}(x;\ell) - \tilde{F}(x;\ell)$, i.e., the error is zero when $\hat{x}$ is at the extreme points of the set $[\ell,u]$.
    \item For $\hat{x} \in[\ell,u]\subseteq \mathbb{R}$, $\breve{F}(u;\hat{x}) = \tilde{F}(u;\hat{x})$ and $\breve{F}(\ell;\hat{x}) - \tilde{F}(\ell;\hat{x})$, i.e., the error is zero when $x$ is at the extreme points of the set $[\ell,u]$.
\end{enumerate}

\end{corollary}

Now we proceed to analyzing the error of the convex hull hypograph, i.e., the minimum of cuts. Let $\breve{F}(x) = \min_{\hat{x}\in \mathcal{X}} \breve{F}(x;\hat{x})$ denote the overestimator obtained by the minimum of approximate cutting planes. While it is clear that $\tilde{F}(x) = F(x)$, there may be a gap between $\breve{F}(x)$ and $F(x)$ for an arbitrary $\hat{x}$, as Figure \ref{fig:convex-hull-hypo} illustrates.

\begin{figure}
	\centering
	\includegraphics[width=0.55\linewidth]{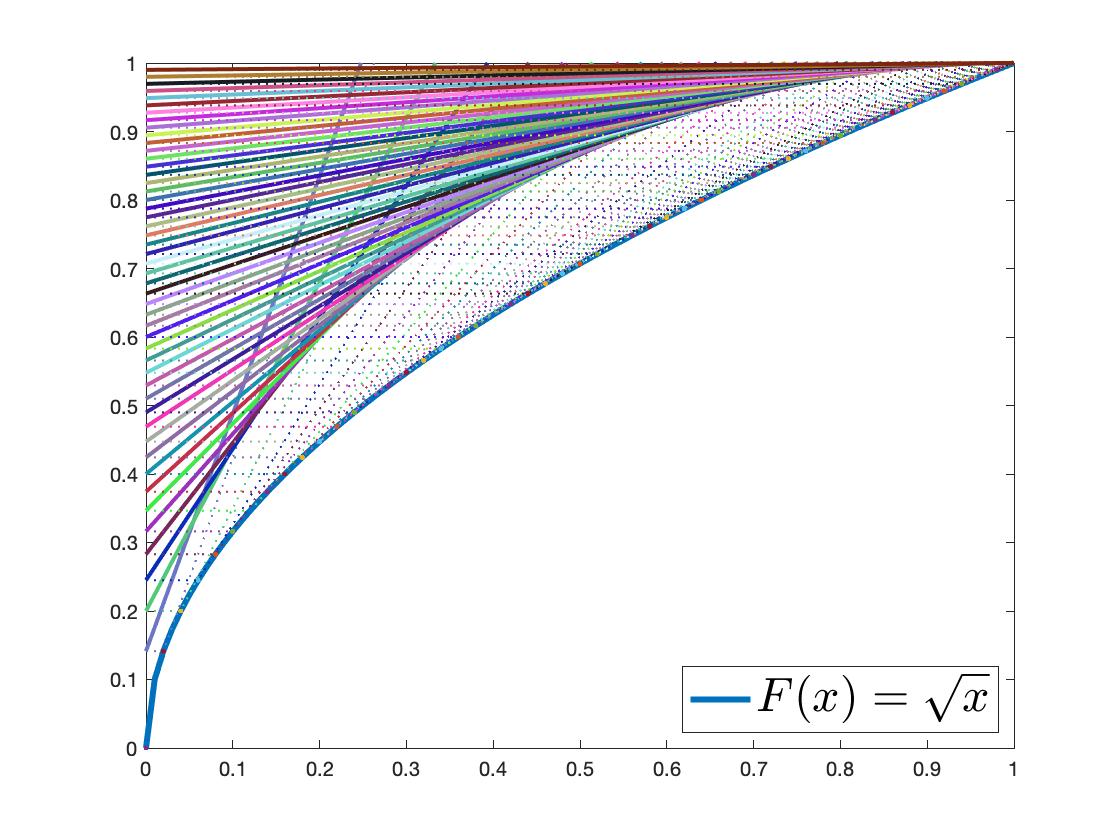}
	\caption{Approximate concave envelope of hypograph of $F(x)=\sqrt{x}$. The curves above the plot of $F(x)=\sqrt{x}$ are concave envelopes of individual cuts, i.e., \eqref{eqn:relu-envelope}.}
	\label{fig:convex-hull-hypo}
\end{figure}

However, Figure \ref{fig:convex-hull-hypo} also indicates that $\breve{F}(x)=F(x)$ at the upper and lower limits of the domain of $F$. The following result, which follows from part 2 of \thref{res:implications-of-max-error}, proves this.

\begin{corollary} \thlabel{res:tight-approx-at-boundary}
For a differentiable non-decreasing continuous DR-submodular function $F$, $F(\ell) = \breve{F}(\ell)$ and $F(u) = \breve{F}(u)$.
\end{corollary}

\begin{proof} 
By part 2 of \thref{res:implications-of-max-error} we have that $\breve{F}(u)=\tilde{F}(u)$ and $\breve{F}(\ell)=\tilde{F}(\ell)$. Then, we use \eqref{eqn:perfect-estimator-at-xHat} to obtain the result.
\end{proof}

And the following result shows that the approximation error of $\breve{F}$ shrinks as the rectangle it is defined over shrinks.

\begin{proposition} \thlabel{res:error-gap-shrinks}
Let $\breve{F}(x;\ell,u)=\min_{\hat{x}\in \mathcal{X}} \breve{F}(x;\hat{x},\ell,u)$, where $\breve{F}(x;\hat{x},\ell,u)$ is an approximate overestimator of $F$ for support $\hat{x}$ over the domain $[\ell,u]$, be an overestimator of $F$ within the rectangle defined by upper bounds $u$ and lower bounds $\ell$. As $||u-\ell||_1\rightarrow 0$, $\breve{F}(x;\ell,u) - F(y) \rightarrow 0$. 
\end{proposition}

\begin{proof} 
We will first show that the error between an approximate cut an and a first-order estimator, $\breve{F}(x;\hat{x},\ell,u)-\tilde{F}(x;\hat{x})$ goes to zero. Because the right-hand side of \eqref{eqn:error-bound2} is bounded above by $\nabla F(\hat{x})^T(u - \ell)$, we have that $\lim_{||u-\ell||_1} \breve{F}(x;\hat{x},\ell,u)-\tilde{F}(x;\hat{x}) = 0$. Then, because the error $\breve{F}(x;\ell,u)-\tilde{F}(x)$ is bounded by the maximum error of a single cut, $\lim_{||u-\ell||_1} \breve{F}(x;\ell,u)-\tilde{F}(x) = 0$. Since $\tilde{F}(x)=F(x)$ (see \eqref{eqn:perfect-estimator-at-xHat}), the result holds.
\end{proof}

\section{Spatial branch-and-bound algorithm}\label{sec:spatial-bb}

In this section we use the approximate convex hull results from Section \ref{sec:approx-hypograph} within a spatial branch-and-bound (SBB) algorithm for finding a global maximizer of $F$. Our SBB approach uses a cutting plane algorithm to maximize over the approximate hypograph in order to compute upper and lower bounds for sub-rectangles.

\subsection{Approximate Cutting Plane Algorithm} \label{sec:approx-cutting-plane-alg}

Given the fact that $\breve{F}(x)$ forms an approximate convex hull of the hypograph of $F$, it is natural to use $\breve{F}(x;\hat{x})$ as cuts with a cutting-plane algorithm based on the following approximate-hypograph-based formulation of \eqref{opt-problem}.

The approximate cutting plane main problem after $Q$ iterations is:

\begin{subequations}
\begin{alignat}{3}
 & \max_{x\in X} & \eta & \\
 & \textrm{subject to} \quad & \eta & \leq  F(\hat{x}_{q}) + \left(\nabla F(\hat{x}_{q})\odot \frac{u-\hat{x}_{q}}{u - \ell} \right)^T (x-\ell) \quad && q = 1,  \dots,Q. \label{eqn:approx-cut}
\end{alignat}
\label{approx-main-prob}
\end{subequations}

Algorithm \ref{approx-cutting-plane-alg} shows pseudocode for the approximate cutting plane algorithm that optimizes over the approximate concave envelope of $F$ within a rectangular region $[\ell,u]$. A main difference between this algorithm and the exact version is the termination criteria. While the exact version terminates when the upper and lower bounds are sufficiently close, because there can exist a gap between $F(x)$ and $\breve{F}(x)$, such a termination criteria will not lead to convergence for an approximate cutting plane algorithm that add cuts of the form \eqref{eqn:approx-cut}. Hence, we used an alternate criteria that terminates when the same cut has been added more than once (see the {\sc CutsAreUnique} procedure on line \ref{alg-line:while-header} of Algorithm \ref{approx-cutting-plane-alg}).

\begin{algorithm}
	{\sc ApproximateCuttingPlane}($\ell, u$)
	\begin{algorithmic}[1]
		\STATE $q \gets 0$
		\WHILE{{\sc CutsAreUnique}($q$)} \label{alg-line:while-header}
		\STATE Solve main problem \eqref{approx-main-prob} to obtain solution $(\eta^*,x^*)$.
		\STATE $\hat{x}^q \gets x^*$
		\STATE Add a cut to the main problem of the form: $\eta \leq F(\hat{x}_{q}) + \left(\nabla F(\hat{x}_{q})\odot \frac{u-\hat{x}_{q}}{u - \ell} \right)^T (x-\ell)$. \label{alg:approx-add-cut-step}
		\STATE $LB_q \gets F(x^*), UB_q \gets \eta^*$.
		\STATE $q \gets q+1$.
		\ENDWHILE
		\STATE \RETURN LB, UB
	\end{algorithmic}
	\caption{Approximate cutting plane algorithm for submodular maximization.}
	\label{approx-cutting-plane-alg}
\end{algorithm}

The following implication of \thref{res:error-gap-shrinks} shows that the final upper and lower bounds computed by the approximate cutting plane algorithm for a rectangle get tighter as the volume of the rectangle approaches zero.

\begin{corollary} \thlabel{res:approx-cp-gap-shrinks}
Let $UB^*$ and $LB^*$ be the final upper and lower bounds computed by Algorithm \ref{approx-cutting-plane-alg} for the recentangle defined by $u$ and $\ell$. As $||u-\ell||_1\rightarrow 0$, $UB-LB \rightarrow 0$. 
\end{corollary}

\begin{proof} 
Algorithm \ref{approx-cutting-plane-alg} terminates when a duplicate cut is added at iteration $q$, which implies that $\hat{x}^q$ is a maximizer of $\breve{F}(x;\ell,u)$. Thus, at termination we have $UB-LB=\eta^*-F(\hat{x}^q)=\breve{F}(\hat{x}^q;\ell,u)-F(\hat{x}^q)$. Applying \thref{res:error-gap-shrinks} yields the result.
\end{proof}

\subsection{Spatial branch-and-bound implementation}\label{sec:spatial-bb-implementation}

The approximate cutting plane algorithm provides a number of benefits that make it suitable to be used as a subroutine within a spatial branch-and-bound algorithm: 1) it provides upper and lower bounds for a sub-rectangle of the feasible region described by the set $X$ and 2) as \thref{res:approx-cp-gap-shrinks} shows the approximate cutting plane algorithm provides tighter bounds as the rectangles defined by $\ell$ and $u$ get smaller. Algorithm \ref{alg:spatial-bb} shows pseudocode for a spatial branch-and-bound algorithm.

\begin{algorithm}
\begin{algorithmic}[1]
	\STATE $Q' \gets Q_{init}$
	\STATE $BestLB, BestUB \gets$ {\sc ApproximateCuttingPlane}($Q'$)
	\STATE $\mathcal{Q} \gets \{Q'\}$ 
	\WHILE{$(BestUB - BestLB)/BestLB > \epsilon$}
	\STATE Remove some $Q'$ from $\mathcal{Q}$ \label{step:node selection}
	\STATE $Q_1, Q_2 \gets$ {\sc Partition}($Q'$) \label{step:partition}
	\FOR{$i=1,2$}
	\STATE $L(Q_i), U(Q_i) \gets$ {\sc ApproximateCuttingPlane}($Q_i$)
	\IF{$U(Q_i) < BestLB$}
	\STATE Add $Q_i$ to $\mathcal{Q}$
	\STATE $BestLB \gets \max\{BestLB, L(Q_i)\}$
	\ENDIF
	\ENDFOR
	\STATE $BestUB \gets \max\{BestUB, U(Q_1), U(Q_2)\}$
	\ENDWHILE
	\STATE \RETURN $BestLB, BestUB$
\end{algorithmic}
\caption{Spatial branch-and-bound algorithm}
\label{alg:spatial-bb}
\end{algorithm}

In the \textit{node selection} step (line \ref{step:partition}), the sub-rectangle $Q'$ with the largest inherited upper bound is chosen. In the {\sc Partition} step (line \ref{step:partition}) the variable with the largest gap between its upper and lower bounds is selected as $\hat{i}$. Next, the rectangle $Q'=[\ell',u']$ is split into two along the $\hat{i}$ axis: in new sub-rectangle $Q_1$ the variable $\hat{i}$ has lower bound $\ell'_i$ and upper bound $u'_i-\ell'_i/2$ and in $Q_2$ $\hat{i}$ has lower bound $u'_i-\ell'_i/2$ and upper bound $u'_i$. In our implementation of Algorithm \ref{alg:spatial-bb} the inequalities added by the approximate cutting plane algorithm when evaluating a node are inherited to child nodes.

\section{Results}\label{sec:results}
To measure the computational performance of our approach, we sought to test it against the best state-of-the-art method. However, although our spatial-branch-and-bound approach is designed for non-factorable problems with gradient information, there are not many existing solvers for this class of problems. Global optimization solvers are typically designed for one of two types of problems. The first type are \textit{white box} problems for which a complete algebraic representation of the objective function and constraints is known, i.e., the factorability assumption. The second type are \textit{black-box} problems for which no structure is known about the objective function and/or constraints. In addition, because it is assumed that the objective function is expensive to evaluate, methods for solving these problems typically do not use gradients or Hessians.

Thus, to test our method against the best competitor, we ran experiments on factorable problems and used the best available solver for these problems, the BARON commercial solver \citep{sahinidis:baron:21.1.13}.

We ran experiments on two types of factorable problems. The first type, problems that admit a \textit{deterministic formulation}, includes a DR-submodular quadratic function and a continuous version of the uncapacitated maximum covering location problem. The second type of problems require a \textit{scenario-based formulation} to represent an expectation. In these formulations the number objective functions terms is exponential in the number of decision variables. These problems include a continuous version of the capacitated maximum covering location problem and a continuous influence maximization problem.

All code was implemented in Python and run on an Apple MacBook Pro running a 2.9 Ghz 6-core Intel core i9 processor with 32 GB of memory. Linear programming subproblems were solved using Gurobi \citep{gurobi}.

For the spatial branch-and-bound (SBB) algorithm, we terminated the approximate cutting plane subroutine when the upper bound had not improved for 5 iterations, a proxy for duplicate cuts being added. We terminated the SBB algorithm when the the relative gap, computed as $(UB-LB)/LB$, reached 0.05 or less or when a one-hour time limit was exceeded. The approximate cutting plane subproblems were also terminated when the relative gap reached 0.001 or the solution had not improved in the last 5 iterations. For BARON we used a one-hour time limit and 0.05 relative optimality gap tolerance.

\subsection{Deterministic formulations}

\subsubsection{Non-decreasing non-concave quadratic submodular maximization}
To test our approach we generated synthetic constrained non-decreasing, non-concave quadratic submodular maximization problems of the form:

\begin{equation}
	\begin{array}{rrclcl}
		\displaystyle \max & F(x) = h^T x + \frac{1}{2}x^T H x + c\\
		\textrm{s.t.} & Ax \leq b & \\
		& \ell \leq x \leq u &
	\end{array}
	\label{eqns:quadratic}
\end{equation}

where $H$ is a symmetric $n\times n$ real matrix and $A$ is $m\times n$. To ensure the submodularity of $F$, we randomly generate the entries of $H$ from $[-1,0]$. In addition, we also generate the entries of $A$ from $[0,1]$. To ensure that the gradient is non-negative, we set $h=-H^Tu$. Without loss of generality, we set the constant term $c$ to zero. We also set $\ell$ to a vector of zeros, and set the vectors $u$ and $b$ to a vector of ones. 

Figure \ref{fig:quadraticPlots} shows the runtime achieved by BARON and the spatial branch-and-bound algorithm when solving instances of \eqref{eqns:quadratic}.

As the figure shows, both methods solve the problems quickly. However, BARON slightly outperforms SBB. The reason that BARON outperformed SBB is probably because the quadratic objective function in \eqref{eqns:covering-defense} consists of factorable bilinear and the number of bilinear terms is polynomial in the problem size. BARON effectively uses overestimators of the bilinear terms in \eqref{eqns:covering-defense} to compute upper bounds. On the other hand, SBB only uses function evaluations and gradient evaluations to compute upper bounds. Because the number of these terms is manageable BARON outperforms SBB. Thus, for these problems exploiting factorability is more useful than exploiting factorability.

\begin{figure}[H]
     \centering
     \begin{subfigure}[b]{\textwidth}
         \centering
         \includegraphics[scale=0.40]{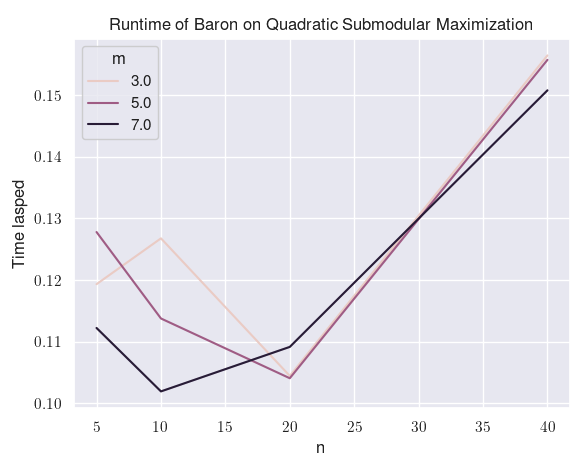}
         \caption{Runtime for BARON solver}
         \label{fig:p110}
     \end{subfigure}\\
    %  \quad
         \begin{subfigure}[b]{\textwidth}
         \centering
         \includegraphics[scale=0.40]{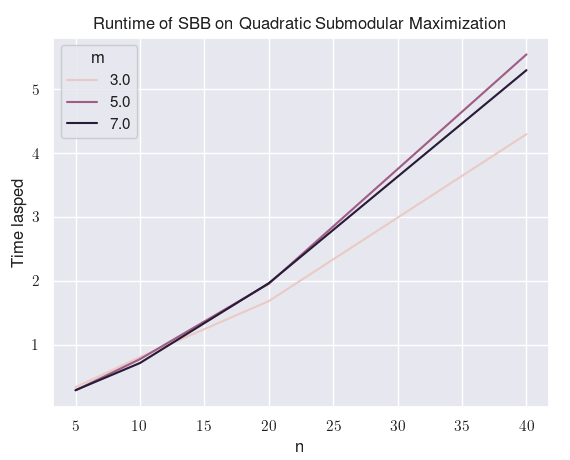}
         \caption{Runtime for SBB algorithm}
         \label{fig:p150}
     \end{subfigure}\\
    \caption{The plots shows the average runtime in seconds achieved by the Spatial Branch and Bound algorithm and BARON for non-decreasing non-concave quadratic submodular maximization. For both SBB and BARON we used the same two randomly-generated instances (consisting of $A$ and $H$) for each row/column combination and computed the average.}
    
    \label{fig:quadraticPlots}
\end{figure}

\subsubsection{Maximum Covering Facility Defense Problem} \label{subsec:fac-defense}
In this subsection we describe the results of experiments on the following maximum covering facility defense problem. Let $J$ be a set of demand points and $[n]$ a set of spatially-located facilities. Facility $i\in [n]$ covers a demand point $j\in J$ if the distance from $i$ to $j$, $d_{ij}$, is no more than some threshold $\bar{d}$. Let $J_i$ be the set of facilities that cover demand point $j$. Then the coverage provided by facilities located at the locations in set $S$ is $f(S) = \left|\bigcup_{i\in S} J_i \right|$; thus $f$ is submodular. Let $g_i$ be a concave function such that $g_i(x_i)$ is the probability that facility $i$ is available given that $x_i$ hardening resources were allocated to $i$. The availability of facilities is assumed to be independent. The feasible region is $X = \{x \,|\, 0\leq x \leq 1, 1^T x \leq b\}$, where $b$ is a budget for hardening resources. Letting $I_j$ be the set of facilities that cover $j$, the maximization problem is:

\begin{equation}
	\label{eqns:covering-defense}
	\begin{array}{rrclcl}
		\displaystyle \max_{x\in X} & F(x) = \mathbb{E}_{S\sim x} f(S) = \displaystyle\sum_{j\in J} w_j \left[ 1 - \prod_{i\in I_j} (1-g_i(x_i)) \right]  
	\end{array}
\end{equation}

Note that because $g$ is concave, $F$ is DR-submodular.

In our experiments we used a distance threshold value of $\bar{d} = 0.2$. We tested an instance of the problem that models the effect of resource allocation using a contest success function with the form $g_i(x_i) = \frac{x_i}{x_i + a_i}$, for some $a_i > 0$, which represents the disruption intensity that facility $i$ is exposed to. In our experiments we chose $a_i = b/|n|$. We tested instances with $|J|=150$ demand points.

Following this setup, Table \ref{fig:problemsSolvedUNCAPFacility} shows the runtime achieved by BARON and spatial branch-and-bound when solving the uncapacitated maximum covering facility defense problem. 

As the table shows, as the problem size increases, BARON solved the problems very quickly while SBB took much longer. The reason that BARON outperformed SBB is probably because the objective function in \eqref{eqns:covering-defense} is factorable and the number of nonlinear terms is linear in the problem size. BARON effectively uses overestimators of the product and quotient terms in \eqref{eqns:covering-defense} to compute upper bounds. On the other hand, SBB only uses function evaluations and gradient evaluations to compute upper bounds. Because the number of these terms increases with $|J|$ and the degree of each term increases with $|I|$, the number of terms is manageable and BARON outperforms SBB. Thus, for these problems exploiting factorability is more useful than exploiting factorability. However, as the results in the next section will show, when the number of nonlinear objective function terms is large, SBB outperforms BARON.
 
\begin{table}[H]
    \centering
    \captionsetup{margin=0.9cm}
    \caption{The average runtime (in seconds) over two randomly-generated instances and the average percent gap achieved by the Spatial Branch and Bound compared with BARON for the uncapacitated problem with different numbers of nodes and budget values for the case in which $g_i(x_i) = \frac{x_i}{x_i + a_i}$ with $a_i = b/|n|$.}
    \label{fig:problemsSolvedUNCAPFacility}
    
   \begin{tabularx}{0.6\linewidth}{@{}Y @{}Y @{}Y @{}Y}
   \hline
  \textbf{$\#$ Nodes ($n$)}  & \textbf{Budget ($b$)} & \textbf{SBB Runtime/Gap} & \textbf{BARON Runtime/Gap} \\ 
  \hline
   5 & 2 & 84.2 & 0.2 \\
%   \hline
   5 & 3 & 30.0 & 0.1 \\
%   \hline
   5 & 4 & 3.7 & 0.1 \\
    6 & 2 & 258.7 & 0.1  \\
%   \hline
   6 & 3 & 125.2 & 0.1  \\
%   \hline
   6 & 4 & 26.2 & 0.1  \\
   7 & 2 & 976.0 & 2.1  \\
%   \hline
   7 & 3 & 635.0 & 4.1  \\
%   \hline
   7 & 4 & 188.6 & 1.1  \\
      8 & 2 & 3457.8 & 0.2  \\
%   \hline
   8 & 3 & 2301.0 & 0.1  \\
%   \hline
   8 & 4 & 1034.8 & 0.1  \\
      9 & 2 & \textcolor{red!60}{6.7\%} & 0.2  \\
%   \hline
   9 & 3 & \textcolor{red!60}{6.1\%} & 0.1  \\
%   \hline
   9 & 4 & \textcolor{red!60}{5.1\%} & 0.1  \\
   \hline
\end{tabularx}
\end{table}

\subsection{Scenario-based Formulations} \label{sec:scenario-based}
\subsubsection{Capacitated Maximum Covering Facility Defense Problem} \label{subsec:cap-fac-defense}
This problem is similar to the uncapacitated version described in the previous section except that each facility $i$ has a capacity $K_i$, hence the coverage function $f(S)$ is computed by solving the following generalized assignment problem:

\begin{equation}
	\begin{array}{rrclcl}
		\displaystyle \max & \multicolumn{3}{l}{	\displaystyle\sum_{i\in S, j\in J_i} y_{ij}}\\
		\textrm{s.t.} & \sum_{j\in J} y_{ij} & \leq & K_i & & \textrm{for } i\in S,\\
		& y_{ij} & \leq & 1  & & \textrm{for } i\in S, j\in J_i,
	\end{array}
	\label{eqns:covering-defense-capacitated}
\end{equation}

where the variable $y_{ij}$ represents the assignment of demand point $j$ to facility $i$. Thus, $f$ is a submodular function. 

Because the facilities have capacities, the compact formulation used for the uncapacitated problem is not valid for the capacitated problem. Thus, we tested the following scenario-based formulation for the problem of allocating hardening resources $x_i$ in order to maximize the expected coverage:

\begin{equation}
	\label{eqns:covering-second-stage}
	\begin{array}{rrclcl}
		\displaystyle \max_{x\in X} & F(x) = \displaystyle\sum_{S\subseteq [n]} f(S) \prod_{i\in S} g_i(x_i) \prod_{i\notin S} (1 - g_i(x_i))
	\end{array}
\end{equation}

We used the same experimental setup as for the uncapacitated case (see \S\ref{subsec:cap-fac-defense}) along with facility capacities of $K_i =|J|/(1-\alpha)|n|$ where $\alpha=0.1$.

Table \ref{fig:problemcapacitatedFacility} shows the runtime achieved by BARON and spatial branch-and-bound when solving the capacitated maximum covering facility defense problem. 

As the table shows, these problems were challenging for both SBB and BARON. This is likely because the objective function \eqref{eqns:covering-second-stage} is a sum over the power set of $[n]$. However, as the table shows, SBB outperformed BARON. While the largest gap for SBB was $18.3\%$, BARON had an optimality gap up to $1997.7\%$ and could only solve the four problems to optimality. The reason that SBB outperformed BARON is probably because the number of nonlinear terms in the objective function \eqref{eqns:covering-second-stage} is exponential in the number of variables ($n$), causing the upper bounds computed by BARON, which are based on overestimators of the individual nonlinear terms, to be weak. Thus, for these problems exploiting submodularity is more useful than exploiting factorability.

\begin{table}[H]
    \centering
    \captionsetup{margin=0.9cm}
    \caption{The table shows the average runtime (over two randomly-generated instances) obtained by Spatial Branch and Bound and the average runtime (or average percent gap between the LB and UB when the runtime exceeds one hour) obtained by BARON for different number of nodes and budget values when solving the capacitated maximum covering facility problem. Text colored in red indicate the percent gap when runtime exceeds one hour.}
    \label{fig:problemcapacitatedFacility}
    
   \begin{tabularx}{0.6\linewidth}{@{}Y @{}Y @{}Y @{}Y}
   \hline
  \textbf{$\#$ Nodes ($n$)}  & \textbf{Budget ($b$)} & \textbf{SBB Runtime/Gap} & \textbf{BARON Runtime/Gap} \\ 
  \hline
   5 & 2 & 217.5 & 62.0 \\
%   \hline
   5 & 3 & 80.9 & 17.6 \\
%   \hline
   5 & 4 & 9.5 & 3.6 \\
    6 & 2 & 1463.9 & \textcolor{red!60}{47.9\%} \\
%   \hline
   6 & 3 & 701.5 & \textcolor{red!60}{1997.7\%} \\
%   \hline
   6 & 4 & 143.8 & 435.4 \\
   7 & 2 & \textcolor{red!60}{7.8\%} & \textcolor{red!60}{240.2\%} \\
%   \hline
   7 & 3 & \textcolor{red!60}{6.3\%} & \textcolor{red!60}{181.0\%} \\
%   \hline
   7 & 4 & 2168.5 & \textcolor{red!60}{133.3\%} \\
      8 & 2 & \textcolor{red!60}{13.6\%} & \textcolor{red!60}{529.6\%} \\
%   \hline
   8 & 3 & \textcolor{red!60}{10.8\%} & \textcolor{red!60}{427.6\%} \\
%   \hline
   8 & 4 & \textcolor{red!60}{8.3\%} & \textcolor{red!60}{325.2\%} \\
      9 & 2 & \textcolor{red!60}{18.3\%} & \textcolor{red!60}{1168.5\%} \\
%   \hline
   9 & 3 & \textcolor{red!60}{15.6\%} & \textcolor{red!60}{963.6\%} \\
%   \hline
   9 & 4 & \textcolor{red!60}{13.0\%} & \textcolor{red!60}{758.8\%} \\
%   \hline
%   9 & 3 & 5188  & \textcolor{red!60}{87\%} \\
   \hline
\end{tabularx}
\end{table}

\subsubsection{Continuous Influence Maximization Problem} \label{subsec:cont-influence-max}
In this problem we study a continuous influence maximization problem on a graph $G=(N,A)$, in which the goal is to allocate marketing resources amongst the nodes in $N$ in order to maximize the expected number of nodes that become influenced during the propagation of influence originating from a set of seed nodes. For each node $i$ let $g_{i}(x_i)$ be a concave function representing the likelihood that node $i$ is initially influenced given that $x_i$ resources are allocated to $i$. After an initial set of seed nodes $S$ become influenced, the influence is propagated through the graph. We used a live-arc representation of the graph in which $G$ represents a random graph in which arcs are randomly present. Let $G_\omega$ represent a realization of $G$, and let $N_{i\omega}$ be the set of nodes that are influenced by seed node $i$ in scenario $\omega$. Thus, the function $f_\omega(S) = |\bigcup_{i\in S} N_{i\omega}|$ computes the total number of nodes influenced by seed set $S$ in realization $\omega$.

\begin{equation}
	\label{eqns:continuous-influence-max}
	\begin{array}{rrclcl}
		\displaystyle \max_{x\in X} & F(x) = \displaystyle\sum_{\omega\in \Omega} \sum_{S\subseteq [n]} f_\omega(S) \prod_{i\in S} g_{i}(x) \prod_{i\notin S} (1 - g_{i}(x))
	\end{array}
\end{equation}

Again, we tested two different cases for the function $g_i$. First, we used the identity function $g_i(x_i) = x_i$, which makes $F$ the standard multilinear extension for a submodular function $f$. Second, we used a contest success function with the form $g_i(x_i) = \frac{x_i}{x_i + a_i}$, for some $a_i > 0$, which represents the disruption intensity that facility $i$ is exposed to. In our experiments we chose $a_i = b/|n|$. In our experiments the number of live-arc realizations was $|\Omega|= 5$. To simulate the spread of influence within the graph we used the independent cascade model \citep{Kempe2003}.

Table \ref{fig:problemIMDefender} shows the runtime achieved by BARON and spatial branch-and-bound when solving the continuous influence maximization problem. The table shows the case where $g_i(x_i) = \frac{x_i}{x_i + a_i}$ with $a_i = b/|n|$ and also shows the case where $g_i(x_i) = x_i$. Again, the problems were challenging for both methods, in part because of the objective function \eqref{eqns:continuous-influence-max} requires summing over the $|\Omega|$ live-arc realizations as well as over the power set of $[n]$. The spatial branch-and-bound algorithm outperformed BARON by achieving a maximum optimality gap of no more than $16.1\%$ on all problems tested. BARON reached an optimality gap of up to $1643.9\%$ and could only solve to optimality when there were no more than 6 nodes in the graph. The reason that SBB outperformed BARON could be because the number of multiplicative terms in \eqref{eqns:continuous-influence-max} is exponential in the number of nodes $n$. Thus, an overestimator based on the combined envelopes of individual multiplicative terms may be weak. On the other hand, SBB uses overestimators that are directly based on the function $F(x)$ rather than individual multiplicative terms. 

\begin{table}[H]
    \centering
    \captionsetup{margin=0.9cm}
    \caption{The table shows the average runtime (over two randomly-generated instances) obtained by Spatial Branch and Bound and the average runtime (or average percent gap between the LB and UB when the runtime exceeds one hour) obtained by BARON for different number of nodes and budget values when solving the continuous influence maximization problem. We show the case when $g_i(x_i) = \frac{x_i}{x_i + a_i}$ with $a_i = b/|n|$ (columns 3 and 4) and also the case when $g_i(x_i) =x_i$ (columns 5 and 6). Text colored in red indicate the percent gap returned after the runtime exceeds one hour. Each instance is a randomly-generated Erd\H{o}s-R\'enyi graph with $p=0.2$ and $2|N|+1$ arcs.}
    \label{fig:problemIMDefender}
    
   \begin{tabularx}{0.6\linewidth}{@{}Y @{}Y @{}Y @{}Y @{}Y @{}Y}
   \hline
  \multirow{2}*\textbf{$\#$ Nodes ($|N|=n$)}  & \multirow{2}*\textbf{Budget ($b$)} & \multicolumn{2}{c}{$g_i(x_i) = \frac{x_i}{x_i + a_i}$} & \multicolumn{2}{c}{$g_i(x_i) =x_i$} \\
  & & SBB Runtime/Gap & BARON Runtime/Gap & SBB Runtime/Gap & BARON Runtime/Gap \\ 
  \hline
   5 & 2 & 106.2 & 85.5 & 0.3 & 56.3  \\
%   \hline
   5 & 3 & 42.9 & 26.1 & 2.0 & 240.3 \\
%   \hline
   5 & 4 & 10.3 & 6.8 & 1.8 & 147.4 \\
    6 & 2 & 849.2 & \textcolor{red!60}{59.9\%} & 1.0 & 2326.8\\
%   \hline
   6 & 3 & 455.9 & 2862.5 & 22.6 & \textcolor{red!60}{97.3\%} \\
%   \hline
   6 & 4 & 146.5 & 520.5 & 12.4 & \textcolor{red!60}{95.0\%} \\
   7 & 2 & \textcolor{red!60}{6.4\%} & \textcolor{red!60}{244.8\%} & 39.0 & \textcolor{red!60}{144.0\%} \\
%   \hline
   7 & 3 & 3520.5 & \textcolor{red!60}{191.9\%} & 120.8 & \textcolor{red!60}{312.2\%} \\
%   \hline
   7 & 4 & 1345.5 & \textcolor{red!60}{142.4\%} & 126.5 & \textcolor{red!60}{433.5\%} \\
      8 & 2 & \textcolor{red!60}{11.4\%} & \textcolor{red!60}{547.2\%} & 1973.9 & \textcolor{red!60}{250.3\%} \\
%   \hline
   8 & 3 & \textcolor{red!60}{9.4\%} & \textcolor{red!60}{432.4\%} & 2288.9 & \textcolor{red!60}{520.6\%} \\
%   \hline
   8 & 4 & \textcolor{red!60}{7.2\%} & \textcolor{red!60}{341.2\%} & 470.0 & \textcolor{red!60}{791.3\%} \\
      9 & 2 & \textcolor{red!60}{16.1\%} & \textcolor{red!60}{1043.4\%} & 1964.3 & \textcolor{red!60}{504.7\%} \\
%   \hline
   9 & 3 & \textcolor{red!60}{13.7\%} & \textcolor{red!60}{881.1\%} & \textcolor{red!60}{10.2\%} & \textcolor{red!60}{1057.2\%} \\
%   \hline
   9 & 4 & \textcolor{red!60}{11.1\%} & \textcolor{red!60}{782.5\%} & \textcolor{red!60}{10.1\%} & \textcolor{red!60}{1643.9\%} \\
%   \hline
%   9 & 3 & 5188  & \textcolor{red!60}{87\%} \\
   \hline
\end{tabularx}
\end{table}

\section{Conclusion} \label{sec:conclusion}
This article presents a new method for finding globally maximal solutions to maximization problems with an objective function that is continuous, non-decreasing, and diminishing returns (DR) submodular. Our method, a spatial-branch-and-bound algorithm, differs from many literature in that it does not require the function to be factorable. Instead, our method only uses function evaluations and gradient evaluations to compute overestimators.

Specifically, our method computes upper bounds via a cutting plane algorithm that maximizes over an approximation of the convex hull of the hypograph of DR-submodular functions.

Although our spatial-branch-and-bound approach is not designed for factorable problems, to ensure the fairest comparison possible, we tested it against the state-of-the-art method for factorable problems, the BARON \citep{sahinidis:baron:21.1.13} commercial solver.

We ran experiments on two types of factorable problems: problems that admit a deterministic formulation and problems that require a scenario-based formulation. In general BARON performed best on the deterministic problems, while our spatial-branch-and-bound method performed best on the scenario-based problems.

The reason that BARON performed best on the deterministic problems is probably because the fact that the objective function is factorable allows uses overestimators of the product and quotient terms to compute upper bounds. As long as the number of such terms is not too large, as is the case for the deterministic problems, BARON's approach is superior to our method, which does not exploit factorability and only uses function evaluations and gradient evaluations. Thus, for these problems it is more useful to exploit the property of factorability than to exploit submodularity.

The reason that our method outperformed BARON on the scenario-based problems is probably because the number of nonlinear terms in the objective function of these problems is exponential in the number of decision variables, causing the upper bounds computed by BARON, which are based on overestimators of the individual nonlinear terms, to be weak. Thus, for these problems it is more useful to exploit the property of submodularity than to exploit factorability.

Regarding the limitations of our approach, although the spatial branch-and-bound algorithm outperformed BARON, it still required significant computation time, even for small problems. This prevented us from testing other types of problems such as D-optimal experimental design and re-weighting in machine learning because these problems often have a larger number of variables. Thus, there is a need for continued improvement in spatial branch-and-bound algorithms for maximizing continuous submodular functions.

Regarding other avenues of future work, in this study we considered the maximization of purely continuous submodular functions, characterized the hypographs of such functions, and developed cutting plane algorithms. However, it would be interesting to investigate the hypograph of mixed-integer continuous submodular functions that are discrete in some variables and continuous in others.

% \section*{Acknowledgments}

%Bibliography
% \bibliographystyle{plain}
\bibliography{biblography.bib}

%%%%%%%%%%%
% APPENDIX
%%%%%%%%%%%
% \clearpage

% \begin{appendices}

% \end{appendices}

\end{document}